\documentclass[oneside,english,american]{article}
\usepackage[T1]{fontenc}
\usepackage[utf8]{inputenc}
\usepackage{geometry}
\usepackage{float}
\usepackage{amstext}
\usepackage{amsthm}
\usepackage{amssymb}
\usepackage{graphicx}

\makeatletter
\theoremstyle{plain}
\newtheorem{thm}{\protect\theoremname}
  \theoremstyle{plain}
  \newtheorem{lem}[thm]{\protect\lemmaname}
  \theoremstyle{plain}
  \newtheorem{prop}[thm]{\protect\propositionname}
  \theoremstyle{definition}
  \newtheorem{problem}[thm]{\protect\problemname}

\date{\today}

\usepackage{babel}

\makeatother

\usepackage{babel}
  \addto\captionsamerican{\renewcommand{\lemmaname}{Lemma}}
  \addto\captionsamerican{\renewcommand{\problemname}{Problem}}
  \addto\captionsamerican{\renewcommand{\propositionname}{Proposition}}
  \addto\captionsamerican{\renewcommand{\theoremname}{Theorem}}
  \addto\captionsenglish{\renewcommand{\lemmaname}{Lemma}}
  \addto\captionsenglish{\renewcommand{\problemname}{Problem}}
  \addto\captionsenglish{\renewcommand{\propositionname}{Proposition}}
  \addto\captionsenglish{\renewcommand{\theoremname}{Theorem}}
  \providecommand{\lemmaname}{Lemma}
  \providecommand{\problemname}{Problem}
  \providecommand{\propositionname}{Proposition}
\providecommand{\theoremname}{Theorem}

\begin{document}

\title{Polynomials with rational generating functions and real zeros}

\author{Tamás Forgács\\
 Khang Tran}
\begin{abstract}
This paper investigates the location of the zeros of a sequence of
polynomials generated by a rational function with a binomial-type
denominator. We show that every member of a two-parameter family consisting
of such generating functions gives rise to a sequence of polynomials
$\{P_{m}(z)\}_{m=0}^{\infty}$ that is eventually hyperbolic. Moreover,
the real zeros of the polynomials $P_{m}(z)$ form a dense subset
of an interval $I\subset\mathbb{R}^{+}$, whose length depends on
the particular values of the parameters in the generating function.\\
 \foreignlanguage{english}{\textbf{MSC:} 30C15, 26C10, 11C08 }
\end{abstract}

\maketitle

\section{Introduction}

The study of the location of zeros of polynomials is one of the oldest
endeavors in mathematics. The prolific mathematical production of
the nineteenth century included a number of advances in this endeavor.
The unsolvability of the general quintic equation together with the
fundamental theorem of algebra led to the consideration that when
it comes to extracting information about the zeros of a complex polynomial
from its coefficients, one should perhaps strive to determine subsets
of $\mathbb{C}$ where the zeros must lie\footnote{Although the publication of \foreignlanguage{english}{\textit{La Géometrie}
predates the early XIXth century by almost two hundred years, from
this perspective, Descartes' rule of signs should be mentioned, as
it gives information on the number of real positive, real negative,
and non-real zeros of a real polynomial}}, rather than looking for the exact location of the zeros. The development
of the Cauchy theory for analytic functions provided some of the classic
machinery suitable for such investigations, including Rouché's theorem,
the argument principle and the Routh-Hurwitz condition for left-half
plane stability. We mention these results not only because they are
powerful tools, but also because they embody a fundamental dichotomy.
Explicit criteria for the location of the zeros of a polynomial in
terms of its coefficients may severely restrict the domain to which
they apply, whereas ubiquitous applicability of a theorem to various
domains may render the result difficult to use. \\
 \indent The Gauss-Lucas theorem, relating the location of the zeros
of $p'(z)$ to those of the polynomial $p(z)$, pioneered a new approach
to an old question: instead of studying the zeros of a function, one
can study the behavior of the zero set of a function under certain
operators. In this light, given the Taylor expansion of a (real) entire
function ${\displaystyle {f(x)=\sum_{k=0}^{\infty}\frac{\gamma_{k}}{k!}x^{k}}}$,
one can interpret $f(x)$ as the result of the sequence $\left\{ \gamma_{k}\right\} _{k=0}^{\infty}$
acting on the function $e^{x}$ by forming a Hadamard product\footnote{We direct the reader to the beautiful works of Hardy \cite{hardy}
and Ostrovskii \cite{ostr} concerning the zero loci of certain entire
functions obtained this way}. Thus, complex sequences have a dual nature: they are coefficients
of `polynomials of infinite degree' (á la Euler), as well as linear
operators on $\mathbb{C}[x]$. G. Pólya and J. Schur's 1914 paper
\cite{ps} was a major mile stone in understanding how sequences (as
linear operators) affect the location of the zeros of polynomials.
More precisely, Pólya and Schur gave a classification of real sequences
that preserve reality of zeros of real polynomials, and initiated
a research program on stability preserving linear operators on circular
domains, which was recently completed by J. Borcea and P. Brändén
(\cite{BB}).\\
 \indent Since the work of Pólya and Schur, the study of sequences
as linear operators on $\mathbb{R}[x]$ has attracted a great amount
of attention. We remark here only that real sequences, when looked
at as operators on $\mathbb{R}[x]$, admit a representation as a formal
power series 
\[
\left\{ \gamma_{k}\right\} _{k=0}^{\infty}\sim\sum_{k=0}^{\infty}Q_{k}(x)D^{k},
\]
where $D$ denotes differentiation, and the $Q_{k}(x)$s are polynomials
with degree $k$ or less (see for example \cite{peetre}). In \cite{tomandrzej}
the first author and A. Piotrowski study the extent to which the `generated'
sequence $\left\{ Q_{k}(x)\right\} _{k=0}^{\infty}$ encodes the reality
preserving properties of the sequence $\{\gamma_{k}\}_{k=0}^{\infty}$,
and find that if this latter sequence is a Hermite-diagonal\footnote{By $\Gamma=\{\gamma_{k}\}_{k=0}^{\infty}$ being a Hermite diagonal
operator we simply mean that $\Gamma[H_{n}(x)]=\gamma_{n}H_{n}(x)$
for all $n$, where $H_{n}(x)$ denotes the $n$th Hermite polynomial. } reality preserving operator, then all of the $Q_{k}(x)$s must have
only real zeros. \\
 \indent The present paper extends the works of S. Beraha, J. Kahane,
and N.~J.~Weiss (\cite{bkw}), A. Sokal (\cite{sokal}) and K. Tran
(\cite{tran}, \cite{tran-1}), by studying a large family of generating
functions which give rise to sequences of polynomials with only real
zeros. Our main result (see Theorem \ref{maintheorem}) concerns a
sequence of polynomials, whose generating function is rational with
a binomial-type denominator. 
\begin{thm}
\label{maintheorem} Let $n,r\in\mathbb{N}$ such that $\max\{r,n\}>1$,
and set $D_{n,r}(t,z):=(1-t)^{n}+zt^{r}$. For all large $m$, the
zeros of the polynomial $P_{m}(z)$ generated by the relation 
\begin{equation}
\sum_{m=0}^{\infty}P_{m}(z)t^{m}=\frac{1}{D_{n,r}(t,z)}\label{genfunc}
\end{equation}
lie on the interval 
\[
I=\left\{
\begin{array}{cc} (0,\infty) & \mbox{ if }n,r\ge2\\
(0,n^{n}/(n-1)^{n-1}) & \text{if} r=1\\
((r-1)^{r-1}/r^{r},\infty) & \text{if} n=1
\end{array}\right.
\]
Furthermore, if $\mathcal{Z}(P_{m})$ denotes the set of zeros of
the polynomial $P_{m}(z)$, then $\bigcup_{m\gg1}\mathcal{Z}(P_{m})$
is dense in $I$. 
\end{thm}
Although this result is asymptotic in nature, we do believe that in
fact all of the generated polynomials have only real zeros. Given
that we have no proof of this claim at this time, we pose this stronger
statement as an open problem (see Problem \ref{prob:1} in Section
\ref{open}).\\
 \indent We close this introduction by noting that if $A(z)$ and
$B(z)$ are any two (non-zero) polynomials with complex coefficients,
then setting $n=1$ and replacing $z$ by $(-1)^{r}A(z)/B(z)^{r}$
and $t$ by $-B(z)t$ in Theorem \ref{maintheorem} reproduces the
main result in a recent paper by the second author: 
\begin{thm}[Theorem 1,\,p.\,879 in \cite{tran-1}]
Let $P_{m}(z)$ be a sequence of polynomials whose generating function
is 
\[
\sum_{m=0}^{\infty}P_{m}(z)t^{m}=\frac{1}{1+B(z)t+A(z)t^{r}},
\]
where $A(z)$ and $B(z)$ are polynomials in $z$ with complex coefficients.
There is a constant $C=C(r)$ such that for all $m>C$, the roots
of $P_{m}(z)$ which satisfy $A(z)\neq0$ lie on a fixed curve given
by 
\[
\Im\left(\frac{B^{r}(z)}{A(z)}\right)=0\qquad\mbox{and}\qquad0\le(-1)^{r}\frac{B^{r}(z)}{A(z)}\le\frac{r^{r}}{(r-1)^{r-1}}
\]
and are dense there as $m\to\infty$. 
\end{thm}
The rest of the paper is organized as follows. In Section \ref{preliminaries}
we present all the preliminary results needed for the proof of Theorem
\ref{maintheorem}. Along with a number of technical lemmas, we prove
a key proposition (Proposition \ref{lem:zerosQ}) concerning estimates
for the relative magnitudes of the zeros (in $t$) of $D(t,z)$. Section
\ref{proofmainthm} is dedicated to the proof of Theorem \ref{maintheorem}.
Before presenting the proof, we illustrate the techniques to be employed
by proving a slightly stronger result for small $n$ and $r$ (Proposition
\ref{illustration}). The paper concludes with Section \ref{open},
where we list some open problems related to our investigations.

\section{Preliminaries}

\label{preliminaries} We establish Theorem \ref{maintheorem} by
showing that for large $m$, the polynomials $P_{m}(z)$ have at least
as many zeros on $I$ as their degree, and since $I\subset\mathbb{R}$,
all zeros of $P_{m}(z)$ must be real when $m\gg1$.Thus, we lose
no generality (in retrospect) by assuming that $z\in\mathbb{R}$,
which we shall do for the remainder of the paper. We start our investigations
with the following 
\begin{lem}
\label{degree} Suppose that the sequence of polynomials ${\displaystyle {\left\{ P_{m}(z)\right\} _{m=0}^{\infty}}}$
is generated by (\ref{genfunc}). Then $\deg P_{m}(z) \leq \lfloor m/r\rfloor$
for all $m$. \end{lem}
\begin{proof}
Rearranging (\ref{genfunc}) yields the equation 
\[
((1-t)^{n}+zt^{r})\sum_{m=0}^{\infty}P_{m}(z)t^{m}=1.
\]
By equating coefficients we see that the polynomial $P_{m}(z)$ satisfies
the recurrence 
\begin{equation}
((1-\Delta)^{n}+z\Delta^{r})P_{m}(z)=0,\label{eq:recurrence}
\end{equation}
where the operator $\Delta$ is defined by $\Delta P_{m}:=P_{m-1}$,
and $\Delta P_{0}=0$. The claim follows. 
\end{proof}
Since we are interested in various $z\in\mathbb{R}$ as potential
zeros of $P_{m}(z)$, we seek to understand how $z$ and the zeros
of $D_{n,r}(t,z)$ (when seen as a polynomial in $t$) correlate.
The next result sheds some light on this question. 
\begin{lem}
Let $n,r\in\mathbb{N}$ be such that $\max\{n,r\}>1$. Suppose $z\in\mathbb{R}$,
and that $t=|t|e^{-i\theta}$ is a zero of $D_{n,r}(t,z)$. Then the
equation 
\begin{equation}
z=\frac{\sin^{n}\theta}{\sin^{n-r}(\phi-\theta)\sin^{r}\phi}\label{eq:ztheta}
\end{equation}
holds, where ${\displaystyle {\phi=\frac{(n-1)\pi+r\theta}{n}}}$. \end{lem}
\begin{proof}
Suppose $z\in\mathbb{R}$. Then $D_{n,r}(t,z)\in\mathbb{R}[t]$, and
consequently if $t=|t|e^{-i\theta}$, $\theta\in\mathbb{R}$, is a
zero of $D_{n,r}(t,z)$, then so is $|t|e^{i\theta}=te^{2i\theta}$.
Rearranging $D_{n,r}(t,z)=D_{n,r}(te^{2i\theta},z)$ yields the equation
\[
\left(\frac{te^{2i\theta}-1}{t-1}\right)^{n}=e^{2ir\theta},
\]
one of whose solutions is 
\begin{equation}
t=\frac{1-e^{2\pi i(n-1)/n}e^{2ir\theta/n}}{e^{2i\theta}-e^{2\pi i(n-1)/n}e^{2ir\theta/n}}.\label{tell}
\end{equation}
We set ${\displaystyle {\phi=\frac{(n-1)\pi+r\theta}{n}}}$ and note
that $\phi-\theta\neq0$. Hence the substitution ${\displaystyle {e^{2\pi i(n-1)/n}e^{2ir\theta/n}=e^{2i\phi}}}$
in equation (\ref{tell}) gives 
\begin{equation}
t=\frac{1-e^{2i\phi}}{e^{2i\theta}-e^{2i\phi}}=e^{-i\theta}\frac{e^{-i\phi}-e^{i\phi}}{e^{i(\theta-\phi)}-e^{-i(\theta-\phi)}}=\frac{\sin\phi}{\sin(\phi-\theta)}e^{-i\theta},\label{eq:t0form}
\end{equation}
and consequently 
\[
(1-t)^{n}=\frac{(-\cos\phi\sin\theta+i\sin\phi\sin\theta)^{n}}{\sin^{n}(\phi-\theta)}=(-1)^{n}\frac{\sin^{n}\theta e^{-in\phi}}{\sin^{n}(\phi-\theta)}=-\frac{\sin^{n}\theta e^{-ir\theta}}{\sin^{n}(\phi-\theta)}.
\]
Solving $D_{n,r}(t,z)=0$ for $z$ thus gives 
\begin{equation}
z=-\frac{(1-t)^{n}}{t^{r}}=\frac{\sin^{n}\theta}{\sin^{n-r}(\phi-\theta)\sin^{r}\phi}.\label{eqn:zintheta}
\end{equation}
The proof is complete. 
\end{proof}
We note that equation (\ref{eq:ztheta}) defines a smooth curve in
$(0,\pi/r)\times\mathbb{R}^{+}$ which, as we shall see, contains
at least $\lfloor m/r\rfloor$ many points with distinct second coordinates
that are all zeros of $P_{m}(z)$ if $m$ is large. For the graphs
of $z$ as a function of $\theta$ when $n=3$, $r=1$, and $r=2$,
see Figure \ref{fig:zthetafig}. 
\begin{figure}[h]
\includegraphics[scale=0.5]{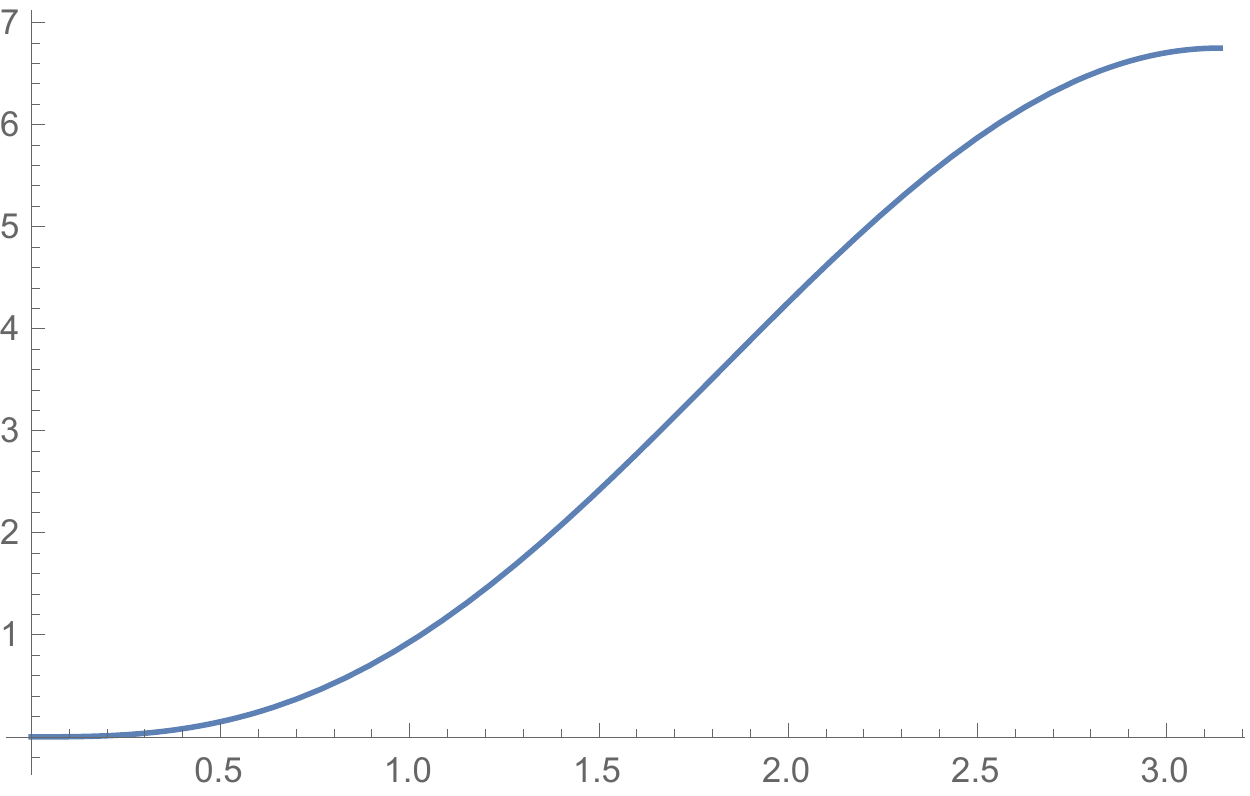}\includegraphics[scale=0.5]{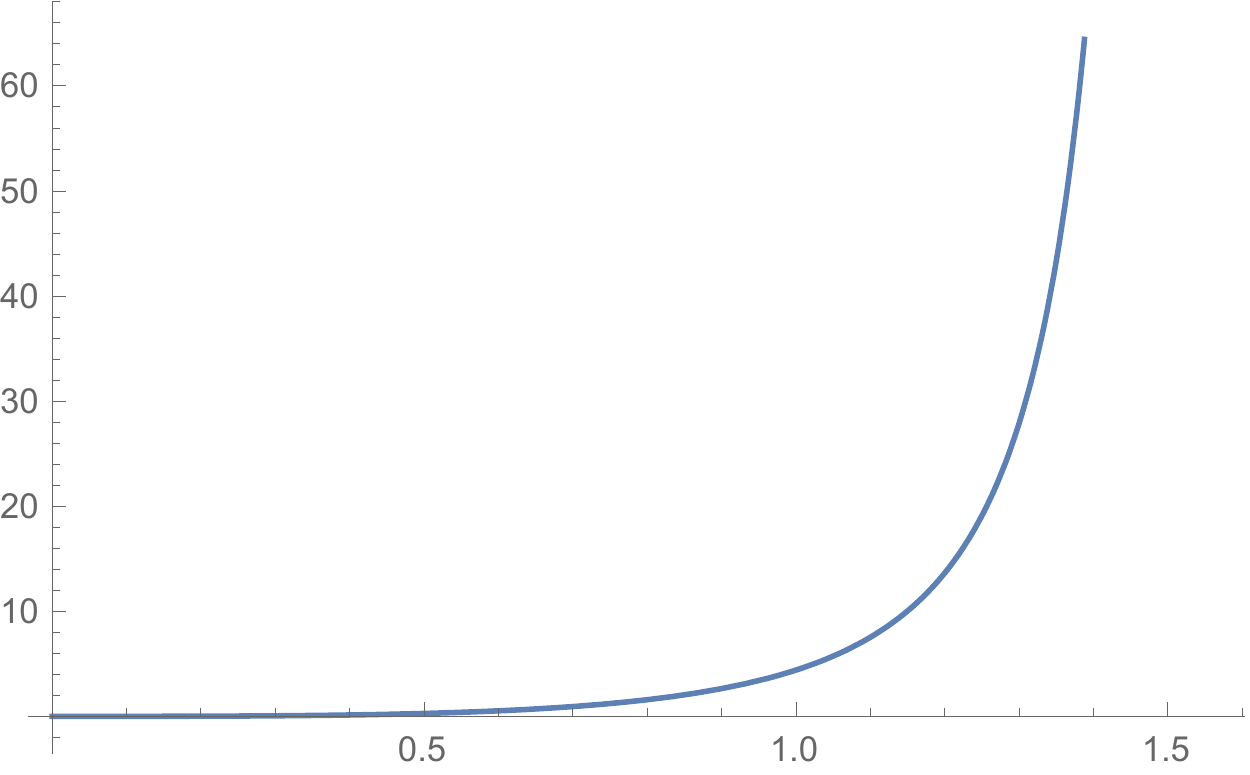}
\caption{\label{fig:zthetafig}The curve $z(\theta)$ when $n=3,r=1$ (left),
and $n=3,r=2$ (right)}
\end{figure}

We next give the proof of two key properties of the curve defined
by (\ref{eq:ztheta}). 
\begin{lem}
\label{lem:zfunctheta} Let $n,r\in\mathbb{N}$ be such that $\max\{n,r\}>1$,
let ${\displaystyle {\phi=\frac{(n-1)\pi+r\theta}{n}}}$, and let
$I$ be the interval as in Theorem \ref{maintheorem}. The function
${\displaystyle {z(\theta)=\frac{\sin^{n}\theta}{\sin^{n-r}(\phi-\theta)\sin^{r}(\phi)}}}$
is increasing on $(0,\pi/r)$, and maps this interval onto $I$.\end{lem}
\begin{proof}
We start by computing three derivatives: 
\begin{eqnarray}
\frac{d}{d\theta}\left(\frac{\sin\theta}{\sin(\phi-\theta)}\right) & = & \frac{\cos\theta\sin(\phi-\theta)-\sin\theta\cos(\phi-\theta)\left(r/n-1\right)}{\sin^{2}(\phi-\theta)}\nonumber \\
 & = & \frac{\sin\phi-r\sin\theta\cos(\phi-\theta)/n}{\sin^{2}(\phi-\theta)},\label{eq:derivfirstfrac}
\end{eqnarray}
\begin{equation}
\frac{d}{d\theta}\left(\frac{\sin\theta}{\sin\phi}\right)=\frac{\cos\theta\sin\phi-r\sin\theta\cos\phi/n}{\sin^{2}\phi},\label{eq:derivsecondfrac}
\end{equation}
and 
\begin{eqnarray}
\frac{d}{d\theta}\left(\frac{\sin(\phi-\theta)}{\sin\phi}\right) & = & \frac{\sin\phi\cos(\phi-\theta)(r/n-1)-r\sin(\phi-\theta)\cos\phi/n}{\sin^{2}\phi}\nonumber \\
 & = & \frac{r\sin\theta/n-\sin\phi\cos(\phi-\theta)}{\sin^{2}\phi}.\label{eq:derivthirdfrac}
\end{eqnarray}
Armed with these calculations we now consider three cases, depending
on the relative sizes of $n$ and $r$.\\
 
\begin{itemize}
\item[\underline{Case $1\ne n=r$}] In this case $z(\theta)$ simplifies to $z(\theta)=\sin\theta/\sin\phi$.
This is an increasing function on $(0,\pi/r)$, since its derivative
(\ref{eq:derivsecondfrac}) is equal to ${\displaystyle {\frac{\sin((n-1)\pi/n)}{\sin((n-1)\pi/n+\theta)}}}$,
a quantity strictly bigger than 0 for $\theta\in(0,\pi/r)$. Finally,
we observe that $z(\theta)$ is continuous on $(0,\pi/r)$, ${\displaystyle {\lim_{\theta\to0}z(\theta)=0}}$,
and that ${\displaystyle {\lim_{\theta\to\pi/r}z(\theta)=+\infty}}$.
We conclude that $z(\theta)$ maps $(0,\pi/r)$ to $(0,\infty)$.
\item[\underline{Case $1\le r<n$}] We write 
\[
z(\theta)=\left(\frac{\sin\theta}{\sin(\phi-\theta)}\right)^{n-r}\left(\frac{\sin\theta}{\sin\phi}\right)^{r},
\]
and note that the inequalities 
\begin{equation}
\pi-\frac{\pi}{n}<\phi<\pi\label{eq:phirange}
\end{equation}
and 
\begin{equation}
\pi-\frac{\pi}{r}<\phi-\theta<\pi-\frac{\pi}{n}\label{eq:phi-thetarange}
\end{equation}
hold for $\theta\in(0,\pi/r)$. We now demonstrate that $\sin\theta/\sin(\phi-\theta)$
and $\sin\theta/\sin\phi$ are both increasing functions of $\theta$
on $(0,\pi/r)$. If $2\le r<n$, then $\cos\theta>0$ and $\cos(\phi-\theta)<0$.
Consequently (\ref{eq:derivfirstfrac}) and (\ref{eq:derivsecondfrac})
are positive. By checking the limits as $\theta\rightarrow0$ and
$\theta\rightarrow\pi/r$, we see that $z(\theta)$ maps $(0,\pi/r)$
onto $(0,\infty)$. If $r=1$ and $\theta\le\pi/2$, then $\cos\theta\ge0$
and $\cos(\phi-\theta)\le0$ and $z(\theta)$ is an increasing function
of $\theta$ for the same reason.\\
 \indent Finally, we treat the case when $r=1$ and $\pi>\theta>\pi/2$.
We write 
\[
\sin\phi=\sin\frac{(n-1)\pi+\theta}{n}=\sin\frac{\pi-\theta}{n},
\]
and note that for any angle $0<\alpha<\pi/2$, the inequality 
\[
\sin\frac{\alpha}{n}>\frac{1}{n}\sin\alpha
\]
holds, as the two sides agree when $\alpha=0$ and the derivative
of the left side is greater than that of the right side when $0<\alpha<\pi/2$.
Thus $\sin\phi>\sin\theta/n$ and (\ref{eq:derivfirstfrac}) is positive.
To show (\ref{eq:derivsecondfrac}) is positive, we note that for
any angle $0<\alpha<\pi/2$, the inequality 
\begin{equation}
\frac{1}{n}\sin\alpha\cos\frac{\alpha}{n}>\cos\alpha\sin\frac{\alpha}{n}\label{eq:trigineq}
\end{equation}
holds since the two sides are equal when $\alpha=0$ and their respective
derivatives satisfy 
\[
\frac{1}{n}\cos\alpha\cos\frac{\alpha}{n}-\frac{1}{n^{2}}\sin\alpha\sin\frac{\alpha}{n}>\frac{1}{n}\cos\alpha\cos\frac{\alpha}{n}-\sin\alpha\sin\frac{\alpha}{n}.
\]
Applying (\ref{eq:trigineq}) with $\alpha=\pi-\theta$ establishes
that (\ref{eq:derivsecondfrac}) is positive. Thus $z(\theta)$ is
an increasing function for all pairs $r,n$ under consideration. We
next compute 
\begin{eqnarray}
\lim_{\theta\rightarrow\pi}\frac{\sin\theta}{\sin(\phi-\theta)} & = & \lim_{\theta\rightarrow\pi}\frac{\sin(\pi-\theta)}{\sin\left((n-1)(\pi-\theta)/n\right)}=\frac{n}{n-1},\label{eq:limitfirstquotient}\\
\lim_{\theta\rightarrow\pi}\frac{\sin\theta}{\sin\phi} & = & \lim_{\theta\rightarrow\pi}\frac{\sin(\pi-\theta)}{\sin((\pi-\theta)/n)}=n.\label{eq:limitsecondquotient}
\end{eqnarray}
Consequently, if $1=r<n$, then $z(\theta)$ maps the interval $(0,\pi)$
to $(0,n^{n}/(n-1)^{n-1})$. 
\item[\underline{Case $1\le n<r$}] We apply arguments akin to those above to the function 
\[
z(\theta)=\left(\frac{\sin\theta}{\sin\phi}\right)^{n}\left(\frac{\sin(\phi-\theta)}{\sin\phi}\right)^{r-n}
\]
and see that $z(\theta)$ is an increasing function of $\theta$ on
$(0,\pi/r)$. If $n>1$, then $z$ maps $(0,\pi/r)$ onto $(0,\infty$).
If $n=1$, we compute 
\begin{eqnarray}
\lim_{\theta\rightarrow0}\frac{\sin\theta}{\sin\phi} & = & \frac{1}{r}\qquad\text{and}\label{eq:limitthirdquotient}\\
\lim_{\theta\rightarrow0}\frac{\sin(\phi-\theta)}{\sin\phi} & = & \frac{r-1}{r},\label{eq:limitfourthquotient}
\end{eqnarray}
and conclude that $z(\theta)$ maps the interval $(0,\pi/r)$ onto
$((r-1)^{r-1}/r^{r},\infty)$. The proof is complete. 
\end{itemize}
\end{proof}
We now reformulate the condition $D_{n,r}(t,z)=0$ by rescaling the
zeros of $D_{n,r}(t,z)$. Although it may seem insignificant at first,
this change in point of view will enable to us derive key magnitude
estimates for these zeros (see Proposition \ref{lem:zerosQ}). These
estimates in turn lay the foundation for the asymptotic analysis,
and with that for the proof of Theorem \ref{maintheorem}, in Section
\ref{proofmainthm}. \\
 \indent We proceed as follows. Suppose $t=|t|e^{-i\theta}$ is a
zero of $D_{n,r}(t,z)$, set 
\begin{eqnarray*}
t_{0} & = & \frac{\sin\phi}{\sin(\phi-\theta)}e^{-i\theta}\qquad\text{(see (\ref{eq:t0form})), and}\\
t_{1} & = & t_{0}e^{2i\theta}=\overline{t_{0}}.
\end{eqnarray*}
After labeling the remaining $\max\{n,r\}-2$ zeros of $D_{n,r}(t,z)$
as $t_{2},t_{3},\ldots,t_{\max\{n,r\}-1}$, write 
\begin{eqnarray*}
q_{k} & = & \frac{t_{k}}{t_{0}},\qquad\text{and}\\
\zeta_{k} & = & q_{k}e^{-i\theta}\qquad\text{for}\quad0\leq k<\max\{n,r\}.
\end{eqnarray*}
With this notation we rewrite $(1-t_{k})^{n}+zt_{k}^{r}=0$ as 
\[
-z=\frac{(1-t_{k})^{n}}{t_{k}^{r}}=\frac{(t_{0}^{-1}-q_{k})^{n}}{q_{k}^{r}}t_{0}^{n-r}.
\]
Combining (\ref{eq:t0form}) with (\ref{eq:ztheta}) we see that $D_{n,r}(t_{k},z)=0$
if and only if 
\[
-\frac{\sin^{n}\theta}{\sin^{n-r}(\phi-\theta)\sin^{r}\phi}=\frac{\left(\sin(\phi-\theta)/\sin\phi-\zeta_{k}\right)^{n}}{\zeta_{k}^{r}}\left(\frac{\sin\phi}{\sin(\phi-\theta)}\right)^{n-r},
\]
or equivalently 
\begin{equation}
\left(\frac{\sin(\phi-\theta)}{\sin\phi}-\zeta_{k}\right)^{n}+\zeta_{k}^{r}\left(\frac{\sin\theta}{\sin\phi}\right)^{n}=0.\label{eq:zetatheta}
\end{equation}
We remark that $k=0,1$ correspond to the two trivial solutions of
(\ref{eq:zetatheta}), namely $\zeta_{0}=e^{-i\theta}$ and $\zeta_{1}=e^{i\theta}$.
The next proposition establishes that these are the solutions of (\ref{eq:zetatheta})
of smallest magnitude.
\begin{prop}
\label{lem:zerosQ} Suppose $n,r\in\mathbb{N}$ are such that $\max\{n,r\}>1$,
$\theta\in(0,\pi/r)$ and $\phi=((n-1)\pi+r\theta)/n$. The polynomial
\begin{equation}
Q(\zeta)=\left(\frac{\sin(\phi-\theta)}{\sin\theta}-\zeta\frac{\sin\phi}{\sin\theta}\right)^{n}+\zeta^{r}\label{eq:Qzeta}
\end{equation}
has exactly two zeros on the unit circle. All other zeros of $Q(\zeta)$
lie outside the closed unit disk $|\zeta|\leq1$. \end{prop}
\begin{proof}
We replace $\zeta$ by $\zeta^{n}$ and consider the zeros\footnote{The change of variables $\zeta\to\zeta^{n}$ maps the (interior, exterior
and the) unit circle to itself (resp). Hence if we prove the proposition
for $\widetilde{Q(\zeta)}$, we simultaneously also get the result
for $Q(\zeta)$.} of 
\[
\widetilde{Q(\zeta)}=\left(\frac{\sin(\phi-\theta)}{\sin\theta}-\zeta^{n}\frac{\sin\phi}{\sin\theta}\right)^{n}+\zeta^{rn}.
\]
Note that if $\zeta$ is a zero of $\widetilde{Q(\zeta)}$, then $\zeta$
is a zero of 
\begin{eqnarray}
R(\zeta) & = & \sin(\phi-\theta)-\zeta^{n}\sin\phi-\omega\zeta^{r}\sin\theta\label{eq:Qzetatransf}\\
 & = & \sin\theta(-\omega\zeta^{r}-\cos\phi)-\sin\phi(\zeta^{n}-\cos\theta),\nonumber 
\end{eqnarray}
where $\omega=\exp\{i(\pi+2\pi k)/n\}$ for some $0\leq k<n$. Consequently,
it suffices to establish the result for $R(\zeta)$.\\
 \foreignlanguage{english}{\textbf{Zeros of $R(\zeta)$ on the unit
circle.} Consider the image of the unit circle $\zeta=e^{i\psi}$,
$0\le\psi<2\pi$, under $R$: 
\[
R(e^{i\psi})=\sin\theta(e^{i(r\psi+\alpha)}-\cos\phi)-\sin\phi(e^{in\psi}-\cos\theta),\qquad0\leq\psi<2\pi,
\]
where 
\[
\alpha=\pi+\frac{\pi+2k\pi}{n},\qquad\text{for some}\quad0\le k<n.
\]
The two curves 
\begin{eqnarray*}
C_{1} & : & \qquad\sin\theta(e^{i(r\psi+\alpha)}-\cos\phi),\qquad0\leq\psi<2\pi\quad\text{and}\\
C_{2} & : & \qquad\sin\phi(e^{in\psi}-\cos\theta)\qquad0\leq\psi<2\pi
\end{eqnarray*}
are circles with radii $r_{1}=\sin\theta,r_{2}=\sin\phi$ and centers
$z_{1}=-\sin\theta\cos\phi,z_{2}=-\sin\phi\cos\theta$ respectively.
(see Figure \ref{fig:Qzeta} when $\cos\theta>0$ and $\cos\phi<0$).
Thus, the equality $R(e^{i\psi})=0$ holds only when these two circles
intersect, and hence the solutions of $R(\zeta)=0$ on the unit circle
must satisfy $n\psi\equiv\pm\theta\mbox{ (mod \ensuremath{2\pi})}$,
or equivalently, $\zeta^{n}=e^{\pm i\theta}$. }\\

\begin{figure}[H]
\begin{centering}
\includegraphics[scale=0.4]{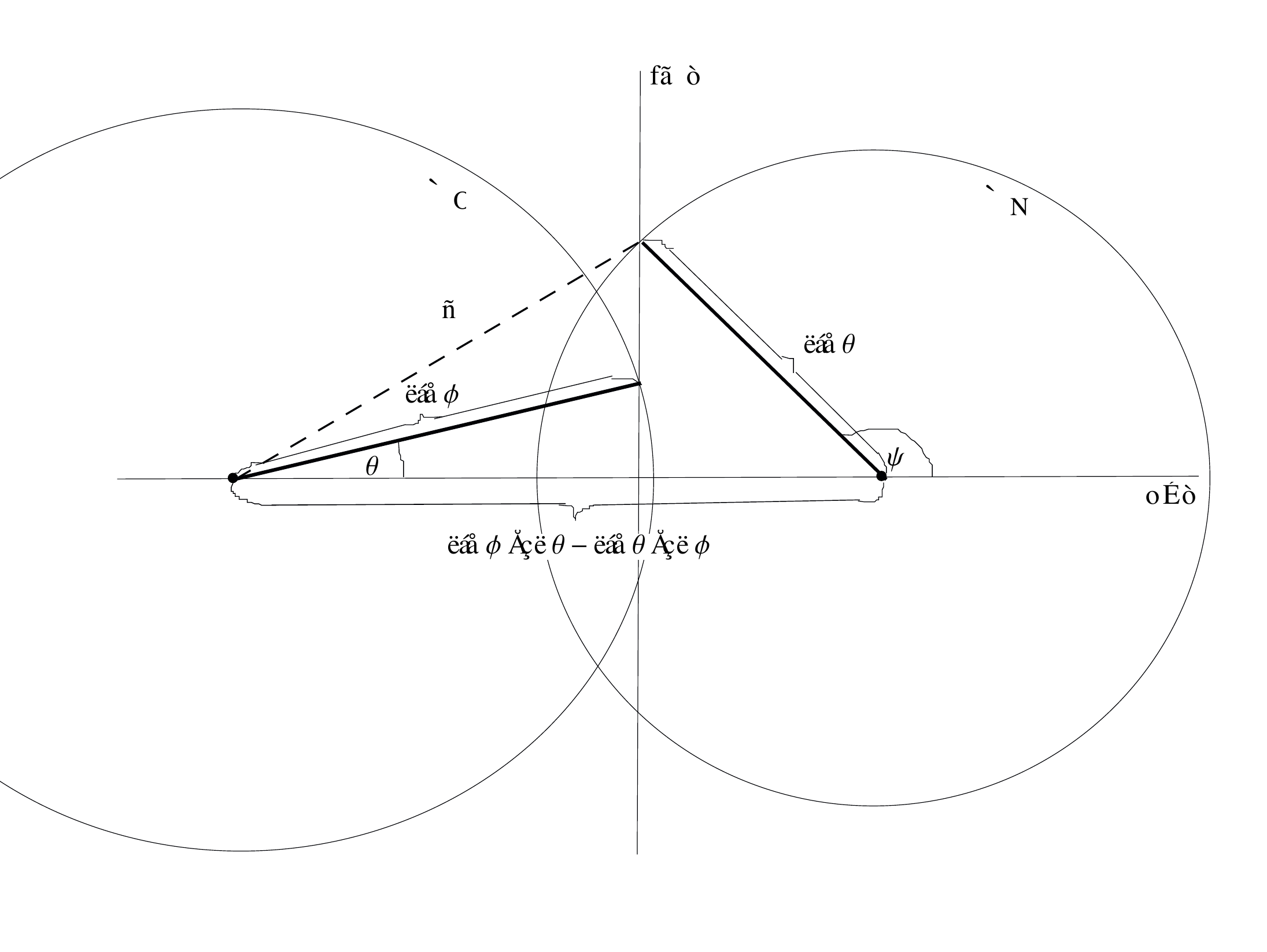} 
\par\end{centering}

\caption{\label{fig:Qzeta0}Generic positioning of the two circles $C_{1}$
and $C_{2}$}
\end{figure}

A priori we don't know where the points of intersection (if any) of
the two circles $C_{1}$ and $C_{2}$ are. A quick calculation using
the law of cosines shows however, that in Figure \ref{fig:Qzeta0},
we must in fact have $x=\sin\phi$, and we obtain the more accurate
Figure \ref{fig:Qzeta}. \\

\begin{figure}[H]
\begin{centering}
\includegraphics[scale=0.4]{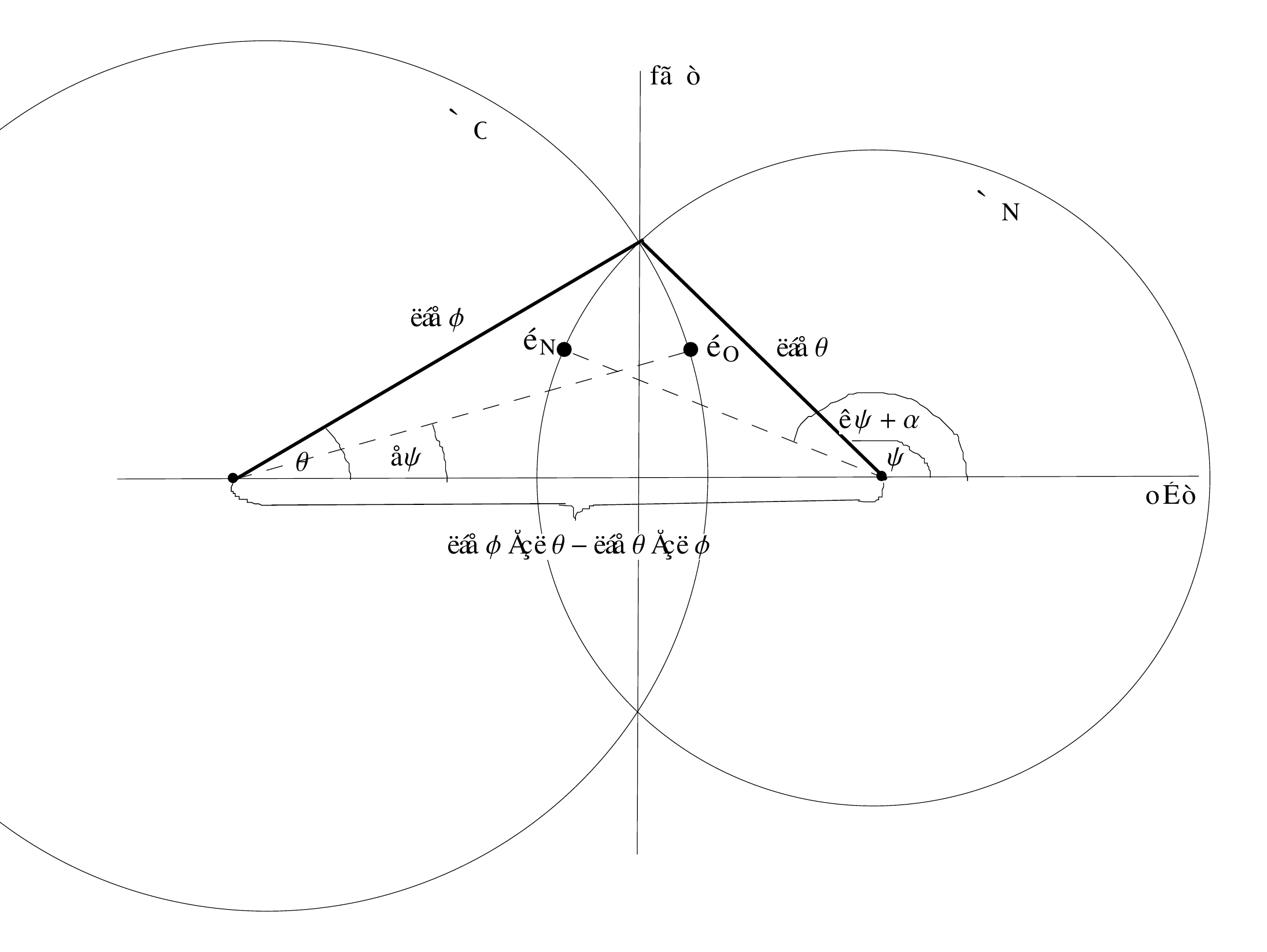} 
\par\end{centering}

\caption{\label{fig:Qzeta}The two circles $C_{1}$ and $C_{2}$}
\end{figure}

\selectlanguage{english}%
\noindent \textbf{Zeros of $R(\zeta)$ inside the unit disk.} We claim
that the loop $\Gamma:=R(e^{i\psi})$, $0\leq\psi<2\pi$, does not
intersect the ray $(-\infty,0)$, and hence the winding number of
$\Gamma$ about any point on $(-\infty,0)$ is $0$. We demonstrate
the claim by treating three different cases, distinguished by whether
$r$, $n$, or neither is equal to one. \\
 If $r\geq2$ and $n\geq2$, then $\cos\theta>0$ and $\cos\phi<0$.
Suppose, by way of contradiction, that $\Gamma$ intersects the ray
$(-\infty,0)$. Then modulo $2\pi$ the inequalities 
\begin{eqnarray*}
 &  & -\theta<n\psi<\theta\quad\text{and}\\
 &  & \phi<r\psi+\alpha<2\pi-\phi
\end{eqnarray*}
must hold. Whence, the points $p_{1}$ and $p_{2}$ (on $C_{1}$ and
$C_{2}$ respectively) corresponding to $\psi$ must lie on the arc
of $C_{1}(C_{2}$ resp) that is inside of $C_{2}(C_{1}$ resp), (see
Figure \ref{fig:Qzeta}). Write 
\begin{eqnarray}
\theta-n\psi & \equiv & \beta\pmod{2\pi}\quad\text{and}\label{eq:congr1}\\
r\psi+\alpha-\phi & \equiv & \gamma\pmod{2\pi}\label{eq:congr2}
\end{eqnarray}
where, given the symmetry about the $x$-axis, we may without loss of generality assume
that $0<\beta\le\theta$ and $0<\gamma\le(\pi-r\theta)/n$. We add
$r$ times the first congruence to $n$ times the second congruence
and obtain 
\[
r\beta+n\gamma\equiv0\mbox{ (mod \ensuremath{2\pi})},
\]
which is impossible since $0<r\beta+n\gamma\le\pi$.\\
If $n=1$, the fact that the zeros of the function 
\[
R(\zeta)=\sin((r-1)\theta)-\zeta\sin r\theta+\zeta^{r}\sin\theta
\]
lie outside the open unit disk is a direct consequence of Lemma 8
in \cite{tran-1}. If $r=1$, we write $R(\zeta)$ as 
\begin{eqnarray*}
R(\zeta) & = & \sin\theta(-\omega\zeta-\cos\phi)-\sin\phi(\zeta^{n}-\cos\theta)\\
 & = & \sin(\pi-\theta)\left(-\omega\zeta+\cos\frac{\pi-\theta}{n}\right)-\sin\frac{\pi-\theta}{n}\left(\zeta^{n}+\cos(\pi-\theta)\right).
\end{eqnarray*}
Applying the substitutions $\omega\zeta\rightarrow\zeta$, $(\pi-\theta)/n\rightarrow\theta$,
and $n\rightarrow r$, reduces this case to the case $n=1$, and completes the
proof. 
\end{proof}
We close this section with a result concerning the multiplicities
of the zeros of $Q(\zeta)$, in addition to their location furbished
by Proposition \ref{lem:zerosQ}.
\begin{lem}
\label{lem:distinctzeros} Let $n,r\in\mathbb{N}$ be such that $\max\{n,r\}>1$,
$0<\theta<\pi/r$, and ${\displaystyle {\phi=\frac{(n-1)\pi+r\theta}{n}}}$.
The zeros of the polynomial 
\[
Q(\zeta)=\left(\frac{\sin(\phi-\theta)}{\sin\theta}-\zeta\frac{\sin\phi}{\sin\theta}\right)^{n}+\zeta^{r}
\]
are distinct, except when 
\begin{enumerate}
\item $n>r>1$ and $r$ is odd, or 
\item $r>n>1$ and $n$ is odd . 
\end{enumerate}
If (1) or (2) hold, then there exists a unique $\theta^{*}\in(0,\pi/r)$
so that $Q(\zeta)$ has one double zero, and all of its remaining
zeros are distinct. \end{lem}
\begin{proof}
Suppose that $\zeta_{*}$ is a multiple zero of $Q(\zeta)$, that
is, $Q(\zeta_{*})=Q'(\zeta_{*})=0$. If $n=r$, $\zeta_{*}$ would
have to satisfy $\zeta_{*}^{n-1}(\sin(\phi-\theta)/\sin\phi)=0$,
and hence $\zeta_{*}=0$. Since $Q(0)\neq0$, we conclude that if
$n=r$, $Q(\zeta)$ has no multiple zeros.\\
 If $n\neq r$, the equations $Q(\zeta_{*})=0=Q'(\zeta_{*})$ imply
that 
\begin{eqnarray}
Q'(\zeta_{*}) & = & -n\frac{\sin\phi}{\sin\theta}\left(\frac{\sin(\phi-\theta)}{\sin\theta}-\zeta_{*}\frac{\sin\phi}{\sin\theta}\right)^{n-1}+r\zeta_{*}^{r-1}\nonumber \\
 & = & \zeta_{*}^{r-1}\left(n\frac{\zeta_{*}\sin\phi}{\sin(\phi-\theta)-\zeta_{*}\sin\phi}+r\right)\label{eq:Qprime}\\
 & = & 0.\nonumber 
\end{eqnarray}
We note that (\ref{eq:Qprime}) has the unique non-zero solution 
\[
\zeta_{*}=-\frac{r}{n-r}\frac{\sin(\phi-\theta)}{\sin\phi}.
\]
Substituting this expression into the equation $Q(\zeta^{*})=0$ yields
\begin{equation}
\frac{\sin^{n}(\phi-\theta)}{\sin^{n}\theta}\left(\frac{n^{n}}{(n-r)^{n}}+(-1)^{r}\frac{r^{r}}{(n-r)^{r}}\frac{\sin^{n}\theta}{\sin^{r}\phi\sin^{n-r}(\phi-\theta)}\right)=0.\label{eq:Qmultzero}
\end{equation}
By Lemma \ref{lem:zfunctheta}, the left hand side of equation (\ref{eq:Qmultzero})
does not vanish for $\theta\in(0,\pi/r)$, unless (1) or (2) hold,
in which case it vanishes for a unique value $\theta^{*}\in(0,\pi/r)$. 
\end{proof}

\section{The proof of Theorem \ref{maintheorem}}

\label{proofmainthm} We now turn our attention to the proof of Theorem
\ref{maintheorem}, which admittedly is technical. As such, that the
reader may benefit from seeing the proof of a special case before
proceeding to the general case. Recall that the generating function
$1/D_{n,r}(t,z)$ of the sequence $\left\{ P_{m}(z)\right\} _{m=0}^{\infty}$
depends on $n,r\in\mathbb{N}$. If $\max\{n,r\}=2$, the result in
Theorem 1 is in fact a consequence of Theorem 1 in \cite{tran}. In
the following proposition we treat the case $\max\{n,r\}=3$. As it
turns out, in this case the conclusions of Theorem \ref{maintheorem}
can be strengthened to $P_{m}(z)$ having only real zeros\footnote{By convention we consider all constant polynomials to have only real
zeros (namely none), including the zero polynomial, which clearly
has infinitely many non-real zeros.} \foreignlanguage{english}{\textit{for all} $m\geq0$.}
\begin{prop}
\label{illustration} Suppose that $n,r\in\mathbb{N}$ be such that
$\max\{n,r\}=3$. Let $D_{n,r}(t,z)=(1-t)^{n}+zt^{r}$, and let $I$
be the real interval in the statement of Theorem \ref{maintheorem}.
Suppose further that 
\[
\sum_{m=0}^{\infty}P_{m}(z)t^{m}=\frac{1}{D_{n,r}(t,z)},
\]
and write $\mathcal{Z}(P_{m})$ for the set of zeros of $P_{m}(z)$.
Then $\mathcal{Z}(P_{m})\subset I$ for all $m\geq0$, and ${\displaystyle {\bigcup_{m=0}^{\infty}\mathcal{Z}(P_{m})}}$
is dense in $I$. \end{prop}
\begin{proof}
Suppose $\theta\in(0,\pi/r)$ and let $z\in I$ be the point so that
$t_{0}=|t_{0}|e^{-i\theta},\,t_{1}=t_{0}e^{2i\theta}$ and $t_{2}\in\mathbb{R}$
are the zeros of $D_{n,r}(t,z)$. Note that by Lemma \ref{lem:distinctzeros},
$t_{0},t_{1}$ and $t_{2}$ are distinct. Thus by partial fraction
decomposition we obtain 
\begin{eqnarray*}
\frac{1}{D_{n,r}(t,z)} & = & -\frac{1}{(t-t_{0})(t_{0}-t_{1})(t_{0}-t_{2})}-\frac{1}{(t-t_{1})(t_{1}-t_{0})(t_{1}-t_{2})}-\frac{1}{(t-t_{2})(t_{2}-t_{0})(t_{2}-t_{1})}\\
 & = & \sum_{m=0}^{\infty}\left(\frac{1}{(t_{0}-t_{1})(t_{0}-t_{2})}\frac{1}{t_{0}^{m+1}}+\frac{1}{(t_{1}-t_{0})(t_{1}-t_{2})}\frac{1}{t_{1}^{m+1}}+\frac{1}{(t_{2}-t_{0})(t_{2}-t_{1})}\frac{1}{t_{2}^{m+1}}\right)t^{m},
\end{eqnarray*}
and consequently, for all $m\geq0$, 
\begin{equation}
P_{m}(z)=\frac{1}{t_{0}^{m+1}(t_{0}-t_{1})(t_{0}-t_{2})}+\frac{1}{t_{1}^{m+1}(t_{1}-t_{0})(t_{1}-t_{2})}+\frac{1}{t_{2}^{m+1}(t_{2}-t_{0})(t_{2}-t_{1})}.\label{eq:pmcubic}
\end{equation}
We set $q_{1}=t_{1}/t_{0}$, $q_{2}=t_{2}/t_{0}$, and divide the
right hand side of (\ref{eq:pmcubic}) by $t_{0}^{m+3}$ to conclude
that $z$ is a zero of $P_{m}(z)$ if and only if 
\begin{equation}
\frac{1}{(1-q_{1})(1-q_{2})}+\frac{1}{q_{1}^{m+1}(q_{1}-1)(q_{1}-q_{2})}+\frac{1}{q_{2}^{m+1}(q_{2}-1)(q_{2}-q_{1})}=0.\label{eq:qexpansion}
\end{equation}
By multiplying the left hand side of (\ref{eq:qexpansion}) by $e^{(m+3)i\theta}$
we obtain the equivalent equation 
\[
-\frac{1}{2i\sin\theta e^{-i(m+1)\theta}(e^{-i\theta}-\zeta_{2})}+\frac{1}{2i\sin\theta e^{i(m+1)\theta}(e^{i\theta}-\zeta_{2})}+\frac{1}{\zeta_{2}^{m+1}(\zeta_{2}^{2}-2\cos\theta\zeta_{2}+1)}=0,
\]
where $\zeta_{2}:=q_{2}e^{-i\theta}\in\mathbb{R}$ since $t_{2}\in\mathbb{R}$.
Combining the first two summands and factoring the expression ${\displaystyle {\frac{1}{\zeta_{2}^{2}-2\zeta_{2}\cos\theta+1}}}$
establishes that equation (\ref{eq:qexpansion}) is equivalent to
\begin{equation}
\frac{\sin(m+2)\theta-\zeta_{2}\sin(m+1)\theta}{\sin\theta}+\frac{1}{\zeta_{2}^{m+1}}=0.\label{eq:thetacrit}
\end{equation}
Finally, using the identity $\sin(m+2)\theta=\sin(m+1)\theta\cos\theta+\cos(m+1)\theta\sin\theta$,
we see that equation (\ref{eq:thetacrit}) holds if and only if $\theta$
is a zero of 
\begin{equation}
R_{m}(\theta)=\frac{(\cos\theta-\zeta_{2})\sin(m+1)\theta}{\sin\theta}+\cos(m+1)\theta+\frac{1}{\zeta_{2}^{m+1}}.\label{eq:Pthethacubic}
\end{equation}
According to Proposition \ref{lem:zerosQ}, $|\zeta_{2}|>1$, and
hence the sign of $R_{m}(\theta)$ alternates at values of $\theta$
for which $\cos(m+1)\theta=\pm1$. Thus, by the Intermediate Value
Theorem, there are at least $\left\lfloor m/r\right\rfloor $ zeros
$\theta$ of $R_{m}(\theta)$ on $(0,\pi/r)$. Lemma \ref{lem:zfunctheta}
in turn implies that each of these zeros $\theta$ yields an distinct
zero $z(\theta)$ of $P_{m}(z)$ on $I$. Since the degree of $P_{m}(z)$
is $\left\lfloor m/r\right\rfloor $, we see that $\mathcal{Z}(P_{m})\subset I$
for all $m\geq0$. Finally, note that the set of solutions to $\cos(m+1)\theta=\pm1$,
$m\geq0$, is dense in $(0,\pi/r)$, which immediately implies that
${\displaystyle {\bigcup_{m=0}^{\infty}\mathcal{Z}(P_{m})}}$ is dense
in $I$, and completes the proof. 
\end{proof}
We continue with some remarks regarding Theorem \ref{maintheorem}
in the case when $\max\{n,r\}>3$ and $m$ is large. If the the zeros
$t_{k}$, $0\le k<\max\{n,r\}$, of $D_{n,r}(t,z)$ are distinct\footnote{Recall that the zeros of $D_{n,r}(t,z)$ and those of $Q(\zeta)$
are scaled copies of each other. Hence, if one has distinct zeros,
so does the other.}, then the partial fraction decomposition of ${\displaystyle {\frac{1}{D_{n,r}(t,z)}}}$
gives 
\begin{eqnarray*}
\sum_{m=0}^{\infty}P_{m}(z)t^{m} & = & \pm\prod_{k=1}^{\max(n,r)-1}\frac{1}{t-t_{k}}\\
 & = & \pm\sum_{k=0}^{\max(n,r)-1}\frac{1}{t-t_{k}}\prod_{l\ne k}\frac{1}{t_{k}-t_{l}}\\
 & = & \pm\sum_{m=0}^{\infty}\left(\sum_{k=0}^{\max(n,r)-1}\frac{1}{t_{k}^{m+1}}\prod_{l\ne k}\frac{1}{t_{k}-t_{l}}\right)t^{m}.
\end{eqnarray*}
Equating coefficients of powers of $t$ in these formal power series
shows that for any $m\geq0$, the equation $P_{m}(z)=0$ is equivalent
to 
\begin{equation}
\sum_{k=0}^{\max(n,r)-1}\frac{1}{t_{k}^{m+1}}\prod_{l\ne k}\frac{1}{t_{k}-t_{l}}=0.\label{eq:Pquotients}
\end{equation}
We multiply equation (\ref{eq:Pquotients}) by $t_{0}^{m+\max\{n,r\}}$
and by $e^{(m+\max\{n,r\})i\theta}$ and conclude that $z$ is a zero
of $P_{m}(z)$ if and only if $\theta$ is a zero of 
\begin{eqnarray}
R_{m}(\theta) & := & \sum_{k=0}^{\max(n,r)-1}\frac{1}{\zeta_{k}^{m+1}}\prod_{l\ne k}\frac{1}{\zeta_{k}-\zeta_{l}}\nonumber \\
 & = & \sum_{k=0}^{\max\{n,r\}-1}\frac{1}{\zeta_{k}^{m+1}Q'(\zeta_{k})},\label{eq:Ptheta}
\end{eqnarray}
where $\zeta_{k}=e^{-i\theta}t_{k}/t_{0}$, $0\le k<\max\{n,r\}-1$,
are the roots of 
\[
Q(\zeta)=\left(\frac{\sin(\phi-\theta)}{\sin\theta}-\zeta\frac{\sin\phi}{\sin\theta}\right)^{n}+\zeta^{r}
\]
with $\phi=((n-1)\pi+r\theta)/n$. Using symmetric reduction we conclude
that $R_{m}(\theta)$ is a real valued continuous function of $\theta$
on $(0,\pi/r)$. \\
 \indent If $D_{n,r}(t,z)$ has zeros of multiplicity greater than
one, then so does $Q(\zeta)$. Thus we must be in either case (1)
or (2) in Lemma \ref{lem:distinctzeros}, and hence there exists a
unique $\theta^{*}\in(0,\pi/r)$ corresponding to the double zero
$\zeta_{*}$ of $Q(\zeta)$. It is clear that $R_{m}(\theta)$ has
a singularity at $\theta^{*}$. We claim that this singularity is
removable. Indeed, if $\zeta_{a}=\zeta_{b}=\zeta_{*}$ for some $0\le a<b<\max\{n,r\}-1$,
then the sum of the two terms of $R_{m}(\theta)$ involving both $\zeta_{a}$
and $\zeta_{b}$ is equal to 
\[
\frac{m+1}{\zeta_{a}^{m+2}}\prod_{l\neq a,b}\frac{1}{\zeta_{a}-\zeta_{l}}.
\]
Thus we may consider $R(\theta)$ to be a real valued continuous function
of $\theta$ on $(0,\pi/r)$, regardless of whether or not the zeros
of $D_{n,r}(t,z)$ are distinct. \\
 \indent Recall that the magnitude estimates of Proposition \ref{lem:zerosQ}
allowed us to dispense with the term in $R_{m}(\theta)$ corresponding
to the third zero of $Q(\zeta)$ in Proposition \ref{illustration}.
In order to make a similar approach work in the general case, we will
need the following calculations concerning the dominating terms in
$R_{m}(\theta)$.\\
 Assume that $\theta\in(0,\pi/r)$, set ${\displaystyle {\phi=\frac{(n-1)\pi+r\theta}{n}}}$,
and recall that $\zeta_{0,1}=e^{\mp i\theta}$ are the two trivial
zeros of $Q(\zeta)$ of minimal magnitude. Using the identity 
\[
\sin(\phi-\theta)-e^{i\theta}\sin\phi=-\sin\theta e^{i\phi},
\]
we rewrite the expression for $Q'(\zeta)$ in (\ref{eq:Qprime}) as
\begin{eqnarray*}
Q'(e^{i\theta}) & = & e^{(r-1)i\theta}\left(-n\frac{\sin\phi}{\sin\theta}e^{i(\theta-\phi)}+r\right).
\end{eqnarray*}
Thus sum of the two terms in (\ref{eq:Ptheta}) corresponding $\zeta_{0}=e^{-i\theta}$
and $\zeta_{1}=e^{i\theta}$ can be written as 
\begin{equation}
2\frac{\Re\left(e^{i(m+1)\theta}Q'(e^{i\theta})\right)}{|Q'(e^{i\theta})|^{2}}=\frac{2}{|Q'(e^{i\theta})|^{2}}\left(A\cos(m+r)\theta-B\sin(m+r)\theta\right),\label{eq:mainterm}
\end{equation}
where 
\begin{eqnarray}
A=A(\theta) & = & -n\frac{\sin\phi}{\sin\theta}\cos(\phi-\theta)+r,\label{eq:Atheta}\\
B=B(\theta) & = & n\frac{\sin\phi}{\sin\theta}\sin(\phi-\theta).\nonumber 
\end{eqnarray}
Just as we did in the proof of Proposition \ref{illustration}, we
will demonstrate that the sign of $R_{m}(\theta)$ changes at points
where $\cos(m+r)\theta=\pm1$. At such points $B(\theta)=0$, so in
order to be able to track the exact number of sign changes in $R_{m}(\theta)$,
it remains to describe the function $A(\theta)$, which we do in the
next lemma. 
\begin{lem}
\label{lem:Atheta}Let $n,r\in\mathbb{N}$ be such that $\max\{n,r\}>1$,
suppose that $\theta\in(0,\pi/r)$ and set $\phi=((n-1)\pi+r\theta)/n$.
The function 
\[
A(\theta)=-n\frac{\sin\phi}{\sin\theta}\cos(\phi-\theta)+r
\]
is strictly positive. \end{lem}
\begin{proof}
If $\cos(\phi-\theta)\le0$ then $A(\theta)>0$ since $\sin\phi>0$
and $\sin\theta>0$ on the interval under consideration. Consequently,
if $r\ge2$ and $n\ge2$, then $A(\theta)>0$.\\
 If $r=1$, then $A(\theta)>0$ if and only if 
\[
\sin\theta-n\sin\phi\cos(\phi-\theta)>0.
\]
Note that the left side of the above inequality is $0$ when $\theta=\pi$,
and it is a decreasing function of $\theta$ on $(0,\pi)$ because
its derivative is negative there: 
\begin{eqnarray*}
 &  & \frac{d}{d\theta}\left(\sin\theta-n\sin\phi\cos(\phi-\theta)\right)\\
 & = & \cos\theta-\cos\phi\cos(\phi-\theta)+n\sin\phi\sin(\phi-\theta)\left(\frac{1}{n}-1\right)\\
 & = & (2-n)\sin\phi\sin(\phi-\theta)\\
 & < & 0.
\end{eqnarray*}
It follows that if $r=1$, then $A(\theta)>0$.\\
 Finally, if $n=1$, then $\phi=r\theta$, and 
\begin{eqnarray*}
A(\theta)\sin\theta & = & r\sin\theta-\sin\phi\cos(\phi-\theta)\\
 & = & (r-1)\sin\theta-\cos\phi\sin(\phi-\theta).
\end{eqnarray*}
Since $A(0)\sin(0)=0$, and 
\[
\frac{d}{d\theta}(A(\theta)\sin(\theta))=(2r-1)\sin(r\theta)\sin((r-1)\theta)>0
\]
for $\theta\in(0,\pi/r)$, we conclude that $A(\theta)>0$ in this
case as well. 
\end{proof}
We are now in position to describe the sign changes of $R_{m}(\theta)$
on the interval $(0,\pi/r)$. 
\begin{prop}
\label{prop:signchangegeneral} Suppose that $n,r\in\mathbb{N}$ are
such that $\max\{n,r\}>1$, and let $R_{m}(\theta)$ be defined as
in equation (\ref{eq:Ptheta}) for $m\geq0$. If ${\displaystyle {\theta_{h}=\frac{h\pi}{m+r},(h=1,\ldots,\left\lfloor m/r\right\rfloor ),}}$
denote the values of $\theta$ in $(0,\pi/r)$ which give $\cos(m+r)\theta=\pm1$,
then 
\begin{itemize}
\item[(i)] $\textrm{sgn}\left(R_{m}(\theta_{h})\right)=(-1)^{h}$, and 
\item[(ii)] $\textrm{sgn}\left(R_{m}(\pi/r^{-})\right)=(-1)^{\left\lfloor m/r\right\rfloor +1}$. 
\end{itemize}
for all $m$ sufficiently large. \end{prop}
\begin{proof}
The proof is broken into three cases as dictated by the asymptotic
behavior of the expressions $\theta_{h}$ as $m$ goes to infinity.
Some of these will stay bounded away from both zero and $\pi/r$,
some will approach $0$, and some will tend to $\pi/r$.

\subsection*{Case 1: $\gamma<\theta<\pi/r-\gamma$ for some small fixed $\gamma$
independent of $m$}

Proposition \ref{lem:zerosQ} implies that if $2\le k<\max(n,r)$,
then $|\zeta_{k}|>1+\epsilon$ for some fixed $\epsilon>0$. Using
Lemma \ref{lem:Atheta} and equation (\ref{eq:mainterm}) we write
\[
R_{m}(\theta)=\frac{2}{|Q'(e^{i\theta})|^{2}}\left(A\cos(m+r)\theta-B\sin(m+r)\theta\right)+\sum_{k=2}^{\max(n,r)-1}\frac{1}{\zeta_{k}^{m+1}Q'(\zeta_{k})},
\]
where the $A(\theta)>0$ and the sum approaches $0$ exponentially
fast when $m\rightarrow\infty$. We invoke Lemma \ref{lem:Atheta}
to conclude that the sign of $R_{m}(\theta_{h})$ is $(-1)^{h}$.

\subsection*{Case 2: $\theta\rightarrow0$ as $m\rightarrow\infty$}

If $n=1$, the zeros $e^{\pm i\theta}$ of $Q(\zeta)$ converge to
$1$, while its remaining zeros converge to those zeros of $T(\zeta)=r-1-r\zeta+\zeta^{r}$
which are greater than one in magnitude\footnote{ $T(\zeta)/(\zeta-1)=\sum_{j=1}^{r-1}\zeta^{j}-(r-1)$ has exactly
one zero in the closed unit disk at $\zeta=1$.}, and hence are uniformly separated from the closed unit disk. Thus,
following an argument similar to the one in Case 1, we see that the
sum in 
\[
R_{m}(\theta)=\frac{2}{|Q'(e^{i\theta})|^{2}}\left(A\cos(m+r)\theta-B\sin(m+r)\theta\right)+\sum_{k=2}^{\max(n,r)-1}\frac{1}{\zeta_{k}^{m+1}Q'(\zeta_{k})},
\]
approaches $0$ exponentially fast as $m\rightarrow\infty$. We
next calculate 
\begin{eqnarray*}
A(\theta) & = & r-\frac{\sin(r\theta)\cos((r-1)\theta)}{\sin\theta}\\
 & = & \left(\frac{r}{3}-r^{2}+\frac{2r^{3}}{3}\right)\theta^{2}+\mathcal{O}(\theta^{4}),
\end{eqnarray*}
and deduce that $A(\theta)\to0$ at a polynomial rate. Consequently,
the sign of $R_{m}(\theta_{h})$ is $(-1)^{h}$ for $m\gg1$.

On the other hand, if $n\ne1$, then $Q(\zeta)$ has $n$ zeros in
a neighborhood of $1$, and the possible remaining $\max\{n,r\}-n$
zeros lie outside the closed unit disk. Assume without loss of generality
that $\zeta_{k}$ is in a neighborhood of $1$ for $0\le k<n$. For
the same reason as in Case 1, the sum 
\[
\sum_{k=n}^{\max(r,n)-1}\frac{1}{\zeta_{k}^{m+1}Q'(\zeta_{k})}
\]
is either $0$, or it approaches $0$ exponentially fast. We next consider
the sum 
\[
\sum_{k=0}^{n-1}\frac{1}{\zeta_{k}^{m+1}Q'(\zeta_{k})}.
\]
Set $\zeta=1+\epsilon$, $\epsilon\in\mathbb{C}$, and note that 
\begin{eqnarray}
\sin\left(\phi-\theta\right)-\zeta\sin\phi & = & \sin\frac{\pi+(n-r)\theta}{n}-\zeta\sin\frac{\pi-r\theta}{n}\nonumber \\
 & = & \sin\frac{\pi}{n}+\left(\cos\frac{\pi}{n}\right)\left(\frac{(n-r)\theta}{n}\right)-(1+\epsilon)\sin\frac{\pi}{n}+(1+\epsilon)\left(\cos\frac{\pi}{n}\right)\left(\frac{r\theta}{n}\right)+\mathcal{O}\left(\theta^{2}\right)\nonumber \\
 & = & \theta\cos\frac{\pi}{n}+\frac{r\theta\epsilon}{n}\cos\frac{\pi}{n}-\epsilon\sin\frac{\pi}{n}+\mathcal{O}\left(\theta^{2}\right).\label{eq:sinediff}
\end{eqnarray}
Thus, if $\zeta$ is a zero of $Q(\zeta)$, then 
\[
0=\left(\theta\cos\frac{\pi}{n}+\frac{r\theta\epsilon}{n}\cos\frac{\pi}{n}-\epsilon\sin\frac{\pi}{n}+\mathcal{O}(\theta^{2})\right)^{n}+\theta^{n}(1+\epsilon)^{r}(1+\mathcal{O}(\theta)).
\]
We deduce that 
\[
\theta\cos\frac{\pi}{n}+\frac{r\theta\epsilon}{n}\cos\frac{\pi}{n}-\epsilon\sin\frac{\pi}{n}=\omega\theta\left(1+\epsilon\right)^{r/n}\left(1+\mathcal{O}\left(\theta\right)\right),
\]
where $\omega$ is an $n$-th root of $-1$. Expanding $(1+\epsilon)^{n/r}$
and solving for $\epsilon$ yields 
\begin{equation}
\epsilon=\frac{\theta\left(\cos(\pi/n)-\omega\right)}{\sin(\pi/n)}+\mathcal{O}\left(\theta^{2}\right).\label{eq:epsilonform}
\end{equation}
Substituting (\ref{eq:sinediff}) and (\ref{eq:epsilonform}) into
the expression (\ref{eq:Qprime}) for $Q'(\zeta)$ we see that if
$\zeta$ is a zero of $Q(\zeta)$, then 
\begin{eqnarray}
Q'(\zeta) & = & \zeta^{r-1}\left(n\frac{\zeta\sin\phi}{\sin(\phi-\theta)-\zeta\sin\phi}+r\right)\nonumber \\
 & = & \zeta^{r-1}\left(\frac{n\zeta\sin(\pi/n)}{\theta\omega}+\mathcal{O}(1)\right).\label{eq:Qprimeasymp}
\end{eqnarray}
Thus, 
\begin{equation}
\sum_{k=0}^{n-1}\frac{1}{\zeta_{k}^{m+1}Q'(\zeta_{k})}=\frac{\theta}{n\sin(\pi/n)}\sum_{k=0}^{n-1}\frac{\omega_{k}}{\zeta_{k}^{m+r}}\left(1+\mathcal{O}(\theta)\right),\label{eq:Pasympzero}
\end{equation}
where $\omega_{k}=e^{(2k-1)\pi i/n}$, $0\le k<n$.

If $\theta\geq\delta/\sqrt{m}$ for some small $\delta$, then equation
(\ref{eq:epsilonform}) gives 
\begin{eqnarray*}
|\zeta_{k}| & \geq & \left|1+\frac{\theta(\cos(\pi/n)-\omega_{k})}{\sin(\pi/n)}\right|\\
 & > & \left|1+\frac{\delta\left(\cos(\pi/n)-\cos\left((2k-1)\pi/n\right)\right)}{2\sqrt{m}\sin(\pi/n)}\right|.
\end{eqnarray*}
From this inequality we deduce that when $m$ is large, 
\[
\sum_{k=2}^{n-1}\frac{1}{\zeta_{k}^{m+1}Q'(\zeta_{k})}
\]
approaches $0$ faster than the polynomial decay of $\frac{2}{|Q'(e^{i\theta})|^{2}}\left(A\cos(m+r)\theta-B\sin(m+r)\theta\right)$.
Thus the sign of $R_{m}(\theta_{h})$ is again $(-1)^{h}$ for $m\gg1$.\\
 On the other hand, if $\theta<\delta/\sqrt{m}$ for $\delta\ll1$,
then 
\begin{eqnarray*}
\zeta_{k}^{m+r} & = & \left(1+\frac{\theta(\cos(\pi/n)-\omega_{k})}{\sin(\pi/n)}+\mathcal{O}(\theta^{2})\right)^{m+r}\\
 & = & \exp\left(\frac{\cos(\pi/n)-\omega_{k}}{\sin(\pi/n)}h\pi\right)\left(1+\mathcal{O}\left(\frac{h^{2}}{m}\right)\right).
\end{eqnarray*}
We select $\delta$ small enough so that the sign of $R_{m}(\theta_{h})$
is the same as the sign of 
\[
\sum_{k=0}^{n-1}\omega_{k}\exp^{-1}\left(\frac{\cos(\pi/n)-\omega_{k}}{\sin(\pi/n)}h\pi\right),
\]
which in turn is determined by the sign of 
\[
\sum_{k=0}^{n-1}\omega_{k}\exp\left(\frac{h\pi\omega_{k}}{\sin(\pi/n)}\right),
\]
since ${\displaystyle {\exp^{-1}\left(\cos(\pi/n)/\sin(\pi/n)\right)>0}}$.
The next lemma establishes that the sign of the expression in question
is $(-1)^{h}$, thereby completing Case 2. 
\begin{lem}
\label{aestimates} Suppose that $n\geq2$. Then for al $h\in\mathbb{N}$,
the sign of the expression 
\begin{equation}
\sum_{k=0}^{n-1}\omega_{k}\exp^{-1}\left(\frac{\cos(\pi/n)-\omega_{k}}{\sin(\pi/n)}h\pi\right),\label{eq:expsum}
\end{equation}
is $(-1)^{h}$. \end{lem}
\begin{proof}
Note that the sum of the first two terms of (\ref{eq:expsum}) is
\[
2(-1)^{h}\cos\left(\frac{\pi}{n}\right),
\]
whereas the sum of the remaining $n-2$ terms is at most 
\begin{equation}
\sum_{k=2}^{n-1}\exp^{-1}\left(\frac{\cos(\pi/n)-\cos((2k-1)\pi/n)}{\sin(\pi/n)}h\pi\right),\label{eq:smallsum}
\end{equation}
a sum which is largest when $h=1$. We use a computer to check that
when $n\le85$ and $h=1$, (\ref{eq:smallsum}) is less than $|2\cos(\pi/n)|$,
hence the sign of (\ref{eq:expsum}) is $(-1)^{h}$. For $n>85$ we
bound (\ref{eq:smallsum}) as follows: 
\begin{eqnarray*}
\sum_{k=2}^{n-1}\exp^{-1}\left(\frac{\cos(\pi/n)-\cos((2k-1)\pi/n)}{\sin(\pi/n)}h\pi\right) & = & \sum_{k=2}^{n-1}\exp^{-1}\left(2\frac{h\pi}{\sin(\pi/n)}\sin\frac{k\pi}{n}\sin\frac{(k-1)\pi}{n}\right)\\
 & < & 2\sum_{k=2}^{\left\lfloor n/2\right\rfloor }\exp^{-1}\left(2h\pi\sin\frac{k\pi}{n}\right)\\
 & \stackrel{(i)}{<} & 2\sum_{k=2}^{\left\lfloor n/2\right\rfloor }\exp^{-1}\left(\frac{hk\pi^{2}}{n}\right)\\
 & \stackrel{(ii)}{<} & \frac{2}{e^{h\pi^{2}/n}(e^{h\pi^{2}/n}-1)},
\end{eqnarray*}
where inequality $(i)$ follows from the fact that $2\sin(k\pi/n)>k\pi/n$
when $k<n/2$, and inequality $(ii)$ is obtained by replacing the
partial geometric sum by the whole series. The reader will note that
the inequality 
\[
\frac{2}{e^{h\pi^{2}/n}(e^{h\pi^{2}/n}-1)}<2\cos\frac{\pi}{85}<2\cos\frac{\pi}{n}
\]
holds when $e^{h\pi^{2}/n}>1.62$, or when $n<20h$. Consequently,
when $n<20h$, 
\[
\sum_{k=2}^{n-1}\omega_{k}\exp^{-1}\left(\frac{\cos(\pi/n)-\omega_{k}}{\sin(\pi/n)}h\pi\right)<|2\cos(\pi/n)|
\]
and the sign of (\ref{eq:expsum}) is $(-1)^{h}$. \\
 Consider now the case $h<n/20$. We note that $n>85$ implies that
\[
\sum_{k=0}^{n-1}\omega_{k}\exp\left(\frac{h\pi\omega_{k}}{\sin(\pi/n)}\right)\approx\sum_{k=0}^{n-1}\omega_{k}e^{hn\omega_{k}}.
\]
By expanding the exponential function in a series we obtain 
\begin{eqnarray}
\sum_{k=0}^{n-1}\omega_{k}e^{hn\omega_{k}} & = & \sum_{k=0}^{n-1}\omega_{k}\sum_{j=0}^{\infty}\frac{1}{j!}\left(hn\omega_{k}\right)^{j}\nonumber \\
 & = & \sum_{j=0}^{\infty}\frac{1}{j!}(hn)^{j}\left(\sum_{k=0}^{n-1}\omega_{k}^{j+1}\right).\label{eq:expseries}
\end{eqnarray}
The identity 
\[
\sum_{k=0}^{n-1}\omega_{k}^{a}=\left\{ \begin{array}{cc}
0 & \text{if} n \nmid a\\
n(-1)^{a/n} & \text{if }n\mid a
\end{array} \right.,
\]
simplifies (\ref{eq:expseries}) to 
\[
n\sum_{j=1}^{\infty}(-1)^{j}\frac{1}{(nj-1)!}\left(hn\right)^{jn-1}=:n\sum_{j=1}^{\infty}a_{j}.
\]
It is straightforward to show that $|a_{j}|<|a_{j+1}|$ for $j=0,1,\ldots,h-1$,
and $|a_{j}|>|a_{j+1}|$ for $j=h,h+1,\ldots$. Since the sum $\sum a_{j}$
is alternating, we conclude that 
\[
\left|\sum_{j=0}^{h-1}a_{j}\right|<|a_{h-1}|,\qquad\text{and}\qquad\left|\sum_{j=h+1}^{\infty}a_{j}\right|<|a_{h+1}|.
\]
We now relate the quantities $|a_{h-1}|,|a_{h}|$ and $|a_{h+1}|$.
To this end, we compute 
\[
\left|\frac{a_{j+1}}{a_{j}}\right|=\frac{(hn)^{n}}{\prod_{\ell=0}^{n-1}nj+\ell},
\]
and consequently, 
\begin{eqnarray*}
\frac{|a_{h-1}|+|a_{h+1}|}{|a_{h}|} & = & \frac{\prod_{\ell=0}^{n-1}(n(h-1)+\ell)}{(nh)^{n}}+\frac{(hn)^{n}}{\prod_{\ell=0}^{n-1}(nh+\ell)}\\
 & = & \prod_{\ell=0}^{n-1}\left(1-\frac{n-\ell}{n}\frac{1}{h}\right)+\frac{1}{\prod_{\ell=1}^{n-1}\left(1+\frac{\ell}{nh}\right)}\\
 & = & P_{1}+\frac{1}{P_{2}}.
\end{eqnarray*}
Taking the natural logarithm of $P_{1}$ we obtain 
\begin{eqnarray*}
\ln(P_{1}) & = & \sum_{\ell=0}^{n-1}\ln\left(1-\frac{n-\ell}{n}\frac{1}{h}\right)\\
 & \stackrel{\ast}{<} & \sum_{\ell=0}^{n-1}\left(-\frac{n-\ell}{n}\frac{1}{h}\right)\\
 & = & -\frac{1}{nh}\frac{n(n+1)}{2}\\
 & = & -\frac{n+1}{2h},
\end{eqnarray*}
where the starred inequality follows from the fact that $\ln x\leq x-1$
for all $x>0$, with equality only when $x=1$. We conclude that $P_{1}<\exp(-\frac{n+1}{2h})$.
On the other hand, taking the natural logarithm of $P_{2}$ gives
\begin{eqnarray*}
\ln(P_{2}) & = & \sum_{\ell=1}^{n-1}\ln\left(1+\frac{\ell}{nh}\right)\\
 & \stackrel{\ast}{>} & \sum_{\ell=1}^{n-1}\frac{\ell}{2nh}\\
 & = & \frac{1}{2nh}\frac{(n-1)n}{2}\\
 & = & \frac{n-1}{4h},
\end{eqnarray*}
where for the starred inequality we used the fact that $\ln x>\frac{x-1}{2}$
for $1<x<2$. Consequently, ${\displaystyle {P_{2}>\exp\left(\frac{n-1}{4h}\right)}}$.
Combining these two results we see that 
\[
P_{1}+\frac{1}{P_{2}}<e^{-\frac{n+1}{2h}}+e^{-\frac{n-1}{4h}}<1,
\]
which in turn implies that $|a_{h-1}|+|a_{h+1}|<|a_{h}|$, and hence
\[
\left|\sum_{j\neq h}a_{j}\right|<|a_{h}|
\]
when $h<n/20$ and $n>85$. Thus, the sign of the expression (\ref{eq:expsum})
is $(-1)^{h}$ in this case as well. 
\end{proof}

\subsection*{Case 3: $\theta\rightarrow\pi/r$ as $m\rightarrow\infty$}

When $\theta\rightarrow\pi/r$ and $r=1$, we observe that with $R(\zeta)$
as in (\ref{eq:Qzetatransf}), the polynomial 
\[
\frac{R(\zeta)}{\sin\theta}=(-\omega\zeta-\cos\phi)-\frac{\sin\phi}{\sin\theta}(\zeta^{n}-\cos\theta)
\]
converges locally uniformly to 
\[
\widetilde{T}(\zeta)=(1-\omega\zeta)-\frac{1}{n}(1+\zeta^{n}).
\]
A quick calculation shows that 
\begin{equation}\label{eq:ttransform}
n \widetilde{T}(\zeta) \big|_{\omega \zeta \to \zeta}=T(\zeta),
\end{equation}
the polynomial we treated at the beginning of case 2. Since the transformations used in (\ref{eq:ttransform}) preserve the location of the zeros in, on, and outside the unit circle, we conclude that as $\theta\rightarrow\pi/r$, $Q(\zeta)$ converges locally
uniformly to a polynomial whose zeros lie outside the closed unit
disk besides the double zeros at $\zeta=-1$. As a result, the sign
of $R_{m}(\theta_{h})$ is again determined by the sign of $\cos(m+r)\theta_{h}$,
and is hence equal to $(-1)^{h}$. \\
 When $r>1$, the polynomial $Q(\zeta)$ has $r$ zeros approaching
the $r$-th roots of $-1$, with the possible remaining $n-r$ zeros
(when $n>r)$ tending to $\infty$. We thus consider the sum 
\[
\sum_{k=0}^{r-1}\frac{1}{\zeta_{k}^{m+1}Q'(\zeta_{k})}
\]
where each $\zeta_{k}$, $0\le k<r$, approaches an $r$-th root of
$-1$. We set $e_{k}=e^{(2k-1)\pi i/r}$ and $\zeta_{k}=e_{k}+\epsilon$,
for some $\epsilon\in\mathbb{C}$, and write 
\[
\eta=\frac{\pi}{r}-\theta=\frac{\pi}{r}-\frac{h\pi}{m+r}=\frac{m+r-hr}{r(m+r)}\pi.
\]
We compute the difference 
\begin{eqnarray*}
\sin\left(\phi-\theta\right)-\zeta\sin\phi & = & \sin\frac{\pi+(n-r)\theta}{n}-\zeta\sin\frac{\pi-r\theta}{n}\\
 & = & \sin\frac{\pi}{r}+\frac{(r-n)\eta}{n}\cos\frac{\pi}{r}-\zeta\frac{r\eta}{n}+\mathcal{O}(\eta^{2}),
\end{eqnarray*}
and use it to rewrite the equation $Q(\zeta_{k})=0$ as 
\begin{eqnarray*}
0 & = & \left(\sin\frac{\pi}{r}+\frac{(r-n)\eta}{n}\cos\frac{\pi}{r}-(e_{k}+\epsilon)\frac{r\eta}{n}\right)^{n}+(\sin\frac{\pi}{r}-\eta\cos\frac{\pi}{r})^{n}(e_{k}+\epsilon)^{r}+\mathcal{O}(\eta^{2})\\
 & = & \left(\sin\frac{\pi}{r}+\frac{(r-n)\eta}{n}\cos\frac{\pi}{r}-e_{k}\frac{r\eta}{n}\right)^{n}-(\sin\frac{\pi}{r}-\eta\cos\frac{\pi}{r})^{n}\left(1+\frac{r\epsilon}{e_{k}}\right)+\mathcal{O}(\eta^{2}+\epsilon^{2}).
\end{eqnarray*}
Taking $n$-th roots yields 
\[
\frac{(r-n)\eta}{n}\cos\frac{\pi}{r}-e_{k}\frac{r\eta}{n}=\frac{r\epsilon}{ne_{k}}\sin\frac{\pi}{r}-\eta\cos\frac{\pi}{r}\left(1+\frac{r\epsilon}{ne_{k}}\right)+\mathcal{O}(\eta^{2}+\epsilon^{2}),
\]
and consequently 
\begin{equation}
\epsilon=\frac{e_{k}(\cos(\pi/r)-e_{k})\eta}{\sin(\pi/r)-\eta\cos(\pi/r)}+\mathcal{O}\left(\eta^{2}\right)=\frac{e_{k}(\cos(\pi/r)-e_{k})\eta}{\sin(\pi/r)}+\mathcal{O}(\eta^{2}).\label{eq:epsilonpi/r}
\end{equation}
From equation (\ref{eq:Qprime}) we also obtain 
\begin{eqnarray*}
Q'(\zeta_{k}) & = & \zeta_{k}^{r-1}\left(n\frac{\zeta_{k}\sin\phi}{\sin(\phi-\theta)-\zeta\sin\phi}+r\right)\\
 & = & r\zeta_{k}^{r-1}+\mathcal{O}(\eta)
\end{eqnarray*}
which we substitute into the sum under consideration to obtain 
\begin{equation}
\sum_{k=0}^{r-1}\frac{1}{\zeta_{k}^{m+1}Q'(\zeta_{k})}=\frac{1}{r}\sum_{k=0}^{r-1}\frac{1+\mathcal{O}(\eta)}{\zeta_{k}^{m+r}}.\label{eq:Pasymppi/r}
\end{equation}
For the same reasons as in Case 2, it suffices to consider $\eta<\delta/\sqrt{m}$
for small $\delta$. In this case we have the approximation 
\begin{eqnarray*}
\zeta_{k}^{m+r} & = & e_{k}^{m+r}\left(1+\frac{\cos(\pi/r)-e_{k}}{\sin(\pi/r)}\eta+\mathcal{O}(\eta^{2})\right)^{m+r}\\
 & = & e_{k}^{m+r}\exp\left(\frac{\cos(\pi/r)-e_{k}}{\sin(\pi/r)}(m+r)\eta\right)\left(1+\mathcal{O}\left(m\eta^{2}\right)\right).
\end{eqnarray*}
We choose $\delta$ small so that the sign of $R_{m}(\theta_{h})$
is the same as the sign of 
\begin{equation}
\sum_{k=0}^{r-1}e_{k}^{-m-r}\exp^{-1}\left(\frac{\cos(\pi/r)-e_{k}}{\sin(\pi/r)}\left(\frac{m}{r}-h+1\right)\pi\right).\label{eq:finalest}
\end{equation}
The sum of the first two terms in the above sum is $2(-1)^{h}$ and
hence, for reasons similar to those in Case 2 (i.e. using the same
argument with $m/r-h+1$ in place of $h$), it suffices to consider
$r>85$. For such $r$, arguments entirely analogous to those we gave
in the proof of Lemma \ref{aestimates} establish that the sign of
the sum (\ref{eq:finalest}) is $(-1)^{h}$. The determine the sign
of $R(\pi/r^{-})$, we first note that for $r>85$, the sign of (\ref{eq:finalest})
is the same as that of 
\begin{equation}
\sum_{k=0}^{r-1}e_{k}^{-m-r}e^{(m-rh+r)e_{k}}=\sum_{k=0}^{r-1}e_{k}^{-m-r}\sum_{j=0}^{\infty}\frac{1}{j!}\left((m-rh+r)e_{k}\right)^{j}.\label{eq:doublesum}
\end{equation}
If we let $m=pr+s$, $0\le s<r$, then the double summation on the
right hand side of (\ref{eq:doublesum}) can be rewritten as 
\begin{equation}
r\sum_{j=0}^{\infty}(-1)^{p+1-j}\frac{\left(m+r-rh\right)^{jr+s}}{(rj+s)!}.\label{eq:signpioverr}
\end{equation}
Now the sign of $R(\pi/r^{-})$ is obtained by replacing $h$ by $(m+r)/r^{-}$
in (\ref{eq:finalest}), and thus by setting $j=0$ in (\ref{eq:signpioverr}).
By doing so we conclude that the sign of $R(\pi/r^{-})$ is $(-1)^{p+1}=(-1)^{\lfloor\frac{m}{r}\rfloor+1}$.
The proof is complete. 
\end{proof}
With Proposition \ref{prop:signchangegeneral} at our disposal, we
now put the finishing touches on the proof of Theorem \ref{maintheorem}.
Let $n,r\in\mathbb{N}$ such that $\max\{r,n\}>1$, and suppose that
\[
\sum_{m=0}^{\infty}P_{m}(z)t^{m}=\frac{1}{(1-t)^{n}+zt^{r}}=:\frac{1}{D_{n,r}(t,z)}.
\]
Given $m\in\mathbb{N}$, every zero of $R_{m}(\theta)$ (see (\ref{eq:Ptheta}))
corresponds to a distinct zero $z\in I$ of $P_{m}(z)$. Proposition
\ref{prop:signchangegeneral}, together with the Intermediate Value
Theorem imply that for $m\gg1$, $R_{m}(\theta)$ has at least $\left\lfloor m/r\right\rfloor $
zeros on $(0,\pi/r)$. The conclusions of Theorem \ref{maintheorem}
now follow from degree considerations, along with the density of the
solutions of the equations $\cos(m+r)\theta=\pm1$, $m=0,1,\ldots$,
in the interval $(0,\pi/r)$.

\section{Some open problems}

\label{open} In light of the conclusions of Theorem $\ref{maintheorem}$
and Proposition $\ref{illustration}$, it is natural to ask whether
the zeros of $P_{m}(z)$ lie on $I$ for all $m$. 
\begin{problem}
\label{prob:1} Let $n,r\in\mathbb{N}$ such that $\max\{n,r\}>1$.
We consider the sequence of polynomials $P_{m}(z)$ generated by 
\[
\sum_{m=0}^{\infty}P_{m}(z)t^{m}=\frac{1}{(1-t)^{n}+zt^{r}},
\]
and write $\mathcal{Z}(P_{m})$ for the set of zeros of $P_{m}(z)$.
Show that $\mathcal{Z}(P_{m})\subset I$ for all $m\ge0$, and ${\displaystyle \bigcup_{m=0}^{\infty}\mathcal{Z}(P_{m})}$
is dense in $I$ where 
\[
I=\left\{ \begin{array}{cc}
(0,\infty) & \text{if} n,r\ge2\\
(0,n^{n}/(n-1)^{n-1}) & \text{if }r=1\\
((r-1)^{r-1}/r^{r},\infty) & \text{ if }n=1
\end{array}\right. .
\]

\end{problem}
A natural way to extend the problem is to consider, in place of the
binomial expression $(1-t)^{n}$, \foreignlanguage{english}{\textit{any}
polynomial with real positive zeros (such polynomials also play a
key role in the theory of multiplier sequences). We formalize this
extension in }
\begin{problem}
Let $Q(t)\in\mathbb{R}[t]$ be a real polynomial in $t$ whose zeros
are positive real numbers and $Q(0)>0$. Show that for any integer
$r\ge0$, the zeros of the polynomial $P_{m}(z)$ generated by 
\[
\sum_{m=0}^{\infty}P_{m}(z)t^{m}=\frac{1}{Q(t)+zt^{r}}
\]
lie on the positive real ray. 
\end{problem}
As we have seen in the proof of Lemma $\ref{degree}$, the denominator
of the generating function gives the recurrence relation for $P_{m}(z)$
whereas the numerator gives rise to the set of initial polynomials.
Allowing for more general numerators in the rational generating function
is a natural extension of the current work. With some numerical evidence
in support, we propose the following 
\begin{problem}
Let $Q(t)\in\mathbb{R}[t]$ be a real polynomial in $t$ whose zeros
are positive real numbers and $Q(0)>0$. For any integer $r\ge0$
and any real bivariate polynomial $N(t,z)\in\mathbb{R}[t,z]$ whose
degree in $t$ is less than the degree in $t$ of $Q(t)+zt^{r}$,
we consider the sequence of polynomials $P_{m}(z)$ generated by 
\[
\sum_{m=0}^{\infty}P_{m}(z)t^{m}=\frac{N(t,z)}{Q(t)+zt^{r}}.
\]
Show that there is a fixed constant (depending perhaps on $r$ and
$\deg Q$) $C$ such that for any $m$, the number of zeros of $P_{m}(z)$
outside $(0,\infty)$ is less than $C$. 
\end{problem}
Finally, although in this paper we do not study them explicitly, transformations
$T$ with the property that the sequence of polynomials $\left\{ P_{m}^{T}(z)\right\} _{m=0}^{\infty}$
generated by $T[G(t,z)]$ have only real zeros whenever those generated
by $G(t,z)$ do are of great interest. We see a natural parallel between
these operators, and reality preserving linear operators on $\mathbb{R}[x]$,
and believe that understanding them is key to understanding how polynomial
sequences with only real zeros might be generated. As such, we pose 
\begin{problem}
Classify operators (bi-linear, or otherwise) $T$ on $\mathbb{R}(t,z)$ with
the property that all terms of the sequence of polynomials $\left\{ P_{m}^{T}(z)\right\} _{m=0}^{\infty}$
generated by $T[G(t,z)]$ have only real zeros whenever those generated
by $G(t,z)$ do. Any results regarding this problem would be pioneering,
even if $G(t,z)$ is restricted to be a rational function of the type discussed in the present paper. \end{problem}

\end{document}